\documentclass{amsart}
\usepackage{amssymb,amsmath,latexsym}
\usepackage{amsthm}
\usepackage{fontenc}
\usepackage{amssymb}
\numberwithin{equation}{section}
\newtheorem{theorem}{Theorem}[section]

\newtheorem{lemma}[theorem]{Lemma}

\begin{document}
\author{Alexander E Patkowski}
\title{On asymptotics for lacunary partition functions}

\maketitle
\begin{abstract}We give asymptotic analysis of power series associated with lacunary partition functions. New partition theoretic interpretations of some basic hypergeometric series are offered as examples. \end{abstract}

\keywords{\it Keywords: \rm Quadratic forms; Partitions; Asymptotics}

\subjclass{ \it 2010 Mathematics Subject Classification 11N37, 30B10.}

\section{Introduction}
It is a known result on quadratic forms in number theory that if $r(n)$ is the number of representations of $n$ as the quadratic form $Q(x_1,x_2):=ax_1^2+bx_1x_2+cx_2^2,$ then [1, Theorem 1, $\beta=1$],[6], as $x\rightarrow\infty,$
\begin{equation}R_1(x):=\sum_{n\le x}r(n)\sim C_D x,\end{equation} 
and
\begin{equation}R_2(x):=\sum_{\substack{ n\le x\\ r(n)>0}}1\sim C_D\frac{x}{\sqrt{\log(x)}},\end{equation}
if $Q(x_1,x_2)$ is not negative definite. Here $C_D$ is a positive constant which depends on the discriminant $D=b^2-4ac.$ We will use the notations $f(x)=O(g(x))$ to mean $|f(x)|\le Kg(x)$ for a positive constant $K,$ and $f(x)\sim g(x)$ to imply $\lim_{x\rightarrow x_0}f(x)/g(x)=1.$ The formula (1.2) is due to P. Bernays, and was used in studying certain lacunary partition functions connected to quadratic forms in [4]. A lacunary partition function is a partition function that is almost always equal to $0.$
\par The goal of this paper is to study the asymptotic behavior of the power series,
\begin{align*} \sum_{n\ge0}a(n)z^n, &\hspace{1cm} \sum_{\substack{ n\ge 0\\ a(n)>0}} z^n,
\end{align*}
as $z\rightarrow 1,$ where $a(n)$ is a lacunary partition function. This goal is accomplished through use of a result found in a paper by Hardy and Littlewood [3].

\begin{lemma} ([3]) Recall $\Gamma(\delta)$ denotes the classical Gamma function. Let $(a_n)$ be a real sequence, and $\sum_{n\le x}a_n\sim x^{\delta}h(x),$ $\delta$ positive, as $x\rightarrow\infty.$ Suppose $x^{\delta}h(x)$ tends to a positive value or $+\infty$ when $x\rightarrow\infty.$ Then, 
 \begin{equation}\sum_{n\ge0}a_nz^n\sim K\frac{\Gamma(\delta+1)}{(1-z)^{\delta+1}}h(\frac{1}{1-z}).\end{equation}
 as $z\rightarrow1.$
\end{lemma}
Specifically, we will use the Bailey lemma to establish a connection between $a(n)$ and $r(n),$ for some quadratic form $Q(x_1,x_2),$ with an attached weight and inequality on each $(x_1,x_2).$ Subsequently (1.1)--(1.2) may be applied in the form of an $O$ result due to the inequality on each $(x_1,x_2)$ and possible cancellation of terms from the weight. This allows us to apply Lemma 1.1 to establish our theorems. The partition interpretations we provide may be of separate interest, and appear to be new. For a recent article connecting partitions related to positive binary quadratic forms see [7].

\begin{theorem}\label{thm:thm1} Let $p_1(n)$ be the number of partitions of $n$ wherein odd parts appear without gaps, the only even part is one plus the largest odd (if it appears at all), weighted by $-1$ raised to the number of odd parts which appear more than once plus the number of appearances of the even part. Then we have,
$$\sum_{n\ge1}p_1(n)z^n=O\left(\frac{1}{1-z}\right),$$
and
$$\sum_{\substack{ n\ge 1\\ p_1(n)>0}} z^n=O\left(\frac{1}{(1-z)\sqrt{\log(\frac{1}{1-z})}}\right),$$
as $z\rightarrow1.$
\end{theorem}

\begin{theorem}\label{thm:thm2} Let $p_2'(n)$ be the number of partitions of $n$ wherein only one part, say $\lambda_i,$ is guaranteed to appear at least once. All other parts are $<2\lambda_i,$ and outside of possibly $\lambda_i,$ the only other even part is $2\lambda_i,$ which may appear any number of times or not at all. If $p_2(n)$ are those partitions of $n$ counted by $p_2'(n),$ weighted by $-1$ raised to the number of parts outside of $\lambda_i,$ then we have
$$\sum_{n\ge1}p_2(n)z^n=O\left(\frac{1}{1-z}\right),$$
$$\sum_{\substack{ n\ge 1\\ p_2(n)>0}} z^n=O\left(\frac{1}{(1-z)\sqrt{\log(\frac{1}{1-z})}}\right),$$
as $z\rightarrow1.$
\end{theorem}

\begin{theorem} \label{thm:thm3} Let $p_3(n)$ be the number of partitions of $n$ wherein (i) if $\lambda_i$ is the largest even, then all positive evens $<\lambda_i$ appear once, and $\lambda_i$ may appear more than once, (ii) odd parts are $<\lambda_i$ and appear an even number of times or not at all, (iii) the weight is $-1$ raised to the number of appearances of even parts. Then we have,
$$\sum_{n\ge1}p_3(n)z^n=O\left(\frac{1}{1-z}\right),$$
$$\sum_{\substack{ n\ge 1\\ p_3(n)>0}} z^n=O\left(\frac{1}{(1-z)\sqrt{\log(\frac{1}{1-z})}}\right),$$
as $z\rightarrow1.$
\end{theorem}

\section{Proofs of Theorems}

The generally accepted notation for $q$-series is [2] $(x;q)_n=(x)_n:=(1-x)(1-xq)\cdots(1-xq^{n-1}).$ The pair $(\alpha_n(a, q),\beta_n(a,q))$ is called a Bailey pair relative to
$(a,q)$ if
\begin{equation}\beta_n(a,q)=\sum_{0\le i\le n}\frac{\alpha_i(a,q)}{(q;q)_{n-i}(aq;q)_{n+i}}.\end{equation} It is known that [8]
\begin{equation}\sum_{n\ge0}(Y_1)_n(Y_2)_n(aq/Y_1Y_2)^n\beta_n(a,q)\end{equation}
$$=\frac{(aq/Y_1)_{\infty}(aq/Y_2)_{\infty}}{(aq)_{\infty}(aq/Y_1Y_2)_{\infty}}\sum_{n\ge0}\frac{(Y_1)_n(Y_2)_n(aq/Y_1Y_2)^n\alpha_n(a,q)}{(aq/Y_1)_n(aq/Y_2)_n}.$$
We will produce some Bailey pairs from the following.
\begin{lemma} (Lovejoy [5]) If $(\alpha_n(a, q),\beta_n(a,q))$ is a Bailey pair, then $(\alpha'_n(aq,b, q),\beta'_n(aq,b,q))$ is a new Bailey pair, where
\begin{equation}\alpha'_n(aq,b, q)=\frac{(1-aq^{2n+1})(aq/b;q)_n(-b)^nq^{n(n-1)/2}}{(1-aq)(bq)_n}\sum_{0\le j \le n}\frac{(b)_j}{(aq/b)_j}(-b)^{-j}q^{-j(j-1)/2}\alpha_j(a,q)\end{equation}
\begin{equation}\beta'_n(aq,b,q)=\frac{(1-b)}{1-bq^n}\beta_n(a,q).\end{equation}
\end{lemma}
We will need two Bailey pairs to provide the lacunary partition examples in our main theorems.
\begin{lemma} We have the following Bailey pairs relative to $(q,q)$ given by
\begin{equation}\alpha_{1,n}=\frac{(1-q^{2n+1})}{1-q}q^{n(n-1)/2}\sum_{2|j|\le n}(-1)^jq^{j^2}, \end{equation}
\begin{equation}\beta_{1,n}=\frac{2}{(1+q^n)(q)_n(q;q^2)_n},\end{equation}
and 
\begin{equation}\alpha_{2,n}=\frac{(1-q^{2n+1})}{1-q}q^{n(n-1)/2}\sum_{2|j|\le n}(-1)^jq^{-j^2}, \end{equation}
\begin{equation}\beta_{2,n}=\frac{2q^{n(n-1)/2}}{(1+q^n)(q)_n(q;q^2)_n}.\end{equation}
\end{lemma}
\begin{proof} The pair $(\alpha_{1,n},\beta_{1,n})$ is obtained by applying Slater's [8, C(1)], relative to $(1,q),$ $\alpha_{2n+1}=0,$ $\alpha_{2n}=(-1)^nq^{n(3n-1)}(1+q^{2n}),$
$$\beta_n=\frac{1}{(q)_n(q;q^2)_n},$$
to Lemma 2.1 with $a=1,$ and $b=-1.$ The pair $(\alpha_{2,n},\beta_{2,n})$ is obtained by applying Slater's [8, C(5)], relative to $(1,q),$ $\alpha_{2n+1}=0,$ $\alpha_{2n}=(-1)^nq^{n(n-1)}(1+q^{2n}),$
$$\beta_n=\frac{q^{n(n-1)/2}}{(q)_n(q;q^2)_n},$$
to Lemma 2.1 with $a=1,$ and $b=-1.$
\end{proof}

In the case of applying the asymptotic formula (1.1) for $R_1(x)$ we will use the $\delta=1,$ $h(x)=1$ case of Lemma 1.1 given in [9, pg.157] to obtain the first estimate in our theorems. For the asymptotic formula (1.2) for $R_2(x)$ we put $\delta=1$ in Lemma 1.1, and $h(x)=1/\sqrt{\log(x)}.$ Note that $xh(x)$ is $+\infty$ as $x\rightarrow\infty.$ Finally we use the triviality
$$\sum_{\substack{ n\ge 0\\ r(n)>0}} z^n=(1-z)\sum_{n\ge1}R_2(n)z^n.$$
\begin{proof}[Proof of Theorem~\ref{thm:thm1}] We take the Bailey pair $(\alpha_{2,n},\beta_{2,n})$ from Lemma 2.2 and insert it into (2.2) with $Y_1=q,$ $Y_2=q^{1/2},$ and replace $q$ by $q^2$ to get
\begin{equation} 1+2\sum_{n\ge1}\frac{q^{n^2}}{(1+q^{2n})(-q;q^2)_n}=\sum_{n\ge0}q^{n^2}(1+q^{2n+1})\sum_{2|j|\le n}(-1)^jq^{-2j^2}.\end{equation} Note that $q^{n^2}/(q;q^2)_n$ generates odd parts without gaps where parts are $\le 2n-1.$ The partition $p_1(n)$ is generated by the left hand side of (2.9) for $n>0$ upon expanding $1/(1+q^{2n})$ in a power series, and the right hand side is a generating function for not negative-definite quadratic forms corresponding to a fundamental discriminant $D=8.$ The asymptotic estimate now follows from our above analysis with Lemma 1.1 once we take the coefficient of $q^n$ of (2.9) for $n>0$ to get
$$p_1(n)=\sum_{\substack{ n=r^2-2j^2\\ 2|j|\le r}}(-1)^j+\sum_{\substack{ n=(r+1)^2-2j^2\\ 2|j|\le r}}(-1)^j$$
\end{proof}

\begin{proof}[Proof of Theorem~\ref{thm:thm2}]We take the Bailey pair $(\alpha_{1,n},\beta_{1,n})$ from Lemma 2.2 and insert it into (2.2) with $Y_1=q,$ $Y_2=q^{1/2},$ and replace $q$ by $q^2$ to get
\begin{equation} 1+2\sum_{n\ge1}\frac{q^n}{(1+q^{2n})(-q;q^2)_n}=\sum_{n\ge0}q^{n^2}(1+q^{2n+1})\sum_{2|j|\le n}(-1)^jq^{2j^2}.\end{equation}
The left hand side of (2.10) generates $p_2(n)$ for $n>0,$ upon expanding $1/(1+q^{2n})$ as a power series, and noting that only the single part generated by $q^n$ is guaranteed to appear. The result follows from the same method as the previous theorem with the notable difference that the quadratic form is positive definite with fundamental discriminant $D=-8.$
\end{proof}

\begin{proof}[Proof of Theorem~\ref{thm:thm3}] We take the Bailey pair $(\alpha_{1,n},\beta_{1,n})$ from Lemma 2.2 and insert it into (2.2) with $Y_1=q,$ $Y_2\rightarrow\infty,$ and replace $q$ by $q^2$ to get
\begin{equation} 1+2\sum_{n\ge1}\frac{(-1)^nq^{n(n+1)}}{(1+q^{2n})(q^2;q^4)_n}=\sum_{n\ge0}(-1)^nq^{n^2}(1-q^{2n+1})\sum_{2|j|\le n}(-1)^jq^{j^2}.\end{equation}

The left hand side of (2.11) generates $p_3(n)$ for $n>0,$ and the right hand side may be observed to be a weighted form of a generating function for the sum of two squares. To see $p_3(n)$ more clearly, note that $(q^2;q^4)_n=(1-q^{1+1})(1-q^{3+3})\cdots(1-q^{2n-1+2n-1}).$ We also need to write $q^{n(n+1)}=q^{2+4+\cdots +2n}.$ The right hand side gives a generating function for positive binary quadratic forms corresponding to a fundamental discriminant $D=-4.$ With the above analysis from this section the result readily follows.

\end{proof}

\section{Conclusion} 

We believe this article provides more information on lacunary partition functions, or those which have intimate connections to binary quadratic forms. Exact $\sim$ results with the associated constant explicitly computed would be of interest and requires application of more tools on binary quadratic forms.

1390 Bumps River Rd. \\*
Centerville, MA
02632 \\*
USA \\*
E-mail: alexpatk@hotmail.com, alexepatkowski@gmail.com

\begin{thebibliography}{9}
\bibitem{ConcreteMath}
V. Blomer, A. Granville,\emph{Estimates for representation numbers of quadratic forms,} Duke Math.
J. Vol. 135, No 2 (2006), 261--302

\bibitem{ConcreteMath}
G. Gasper and M. Rahman, Basic Hypergeometric Series, Encyclopedia Math. Appl. 35, Cambridge Univ. Press, Cambridge, 1990.


\bibitem{ConcreteMath} G. H. Hardy, J. E. Littlewood, \emph{Tauberian theorems concerning power series and Dirichlet's series whose coefficients are positive,} Proc. London Math. Soc., 13, (1914) pp. 174--191.

\bibitem{ConcreteMath} 
J. Lovejoy, \emph{Lacunary partition functions,} Math. Res. Lett. 9 (2002), 191--198.
\bibitem{ConcreteMath} 
J. Lovejoy, \emph{A Bailey Lattice,} Proc. Amer. Math. Soc. 132 (2004), 1507--1516.
\bibitem{ConcreteMath} 
R. Odoni, \emph{Representations of algebraic integers by binary quadratic forms and
norm forms from full modules of extension fields,} J. Number Theory, 10 (1978), 324--333
\bibitem{ConcreteMath} 
A.E. Patkowski, \emph{Partitions related to positive definite binary quadratic forms,} INTEGERS, Volume 19, A25, (2019).

\bibitem{ConcreteMath} L. J. Slater, \emph{A new proof of Rogers transformations of infinite series,} Proc. London Math.
Soc. (2) 53 (1951), 460--475.

\bibitem{ConcereteMath} E. C. Titchmarsh, \emph{The theory of the Riemann zeta function,} Oxford University Press, 2nd edition, 1986.


\end{thebibliography}
\end{document}